\documentclass[a4paper,11pt]{amsart}

\usepackage{amsfonts,amsmath,amssymb,mathrsfs,amsthm}
\usepackage{hyperref}
\hypersetup{colorlinks, urlcolor = blue, linkcolor = blue, citecolor = red}

\newtheorem{theorem}{Theorem}[section]
\newtheorem{proposition}[theorem]{Proposition}
\newtheorem{lemma}[theorem]{Lemma}
\newtheorem{corollary}[theorem]{Corollary}
\theoremstyle{definition}

\newtheorem{example}[theorem]{Example}
\theoremstyle{remark}
\newtheorem{remark}[theorem]{Remark}

\renewcommand{\geq}{\geqslant}
\renewcommand{\leq}{\leqslant}

\newcommand{\R}{\mathbb{R}}

\newcommand{\E}{\mathbb{E}}

\newcommand{\unit}{e}
\newcommand{\ba}{\mathsf{ba}}
\newcommand{\bbP}{\mathbb{P}}
\newcommand{\clo}[1]{\overline{#1}}
\def\<#1,#2>{\langle #1, #2\rangle}
\DeclareMathOperator{\ext}{ext}
\newcommand{\sft}{\mathsf{t}}

\title[Minimax representation of nonexpansive functions]{Minimax representation of nonexpansive functions and application to zero-sum recursive games}

\author[M. Akian]{Marianne Akian}
\author[S. Gaubert]{St\'ephane Gaubert}
\author[A. Hochart]{Antoine Hochart}
\address{INRIA Saclay-Ile-de-France and CMAP, Ecole polytechnique, Route de Saclay, 91128 Palaiseau Cedex, France}
\thanks{A.~Hochart has been supported by a PhD fellowship of Fondation Math\'ematique Jacques Hadamard (FMJH). The authors are also partially supported by the PGMO programme of EDF and FMJH and by the ANR through the MALTHY project, ANR-13-INSE-0003}
\email{marianne.akian@inria.fr}
\email{stephane.gaubert@inria.fr}
\email{antoine.hochart@polytechnique.edu}

\date{May 24, 2017}

\begin{document}

\begin{abstract}
  We show that a real-valued function on a topological vector space is positively homogeneous
  of degree one and nonexpansive with respect to a weak Minkowski norm if and only if it can be
  written as a minimax of linear forms that are nonexpansive with respect to the same norm.
  We derive a representation of monotone, additively and positively homogeneous functions on
  $L^\infty$ spaces and on $\mathbb{R}^n$, which extends results of Kolokoltsov, Rubinov,
  Singer, and others.
  We apply this representation to nonconvex risk measures and to zero-sum games.
  We derive in particular results of representation and polyhedral approximation for
  the class of Shapley operators arising from games without instantaneous payments
  (Everett's recursive games). 
\end{abstract}



\maketitle

\section{Introduction}

Repeated zero-sum games can be studied by means of dynamic programming operators,
also known as {\em Shapley operators}.
For instance, in the case of a game with finite state space $\{1,\dots,n\}$ and perfect
information, the Shapley operator is a self-map of $\R^n$ given by
\begin{equation}
  \label{eq:ShapleyOperator}
  [f(x)]_i = \inf_{a \in A_i} \sup_{b \in B_{i,a}} \Big\{ r^{ab}_i + \sum_{j=1}^n P^{ab}_{ij} x_j \Big\} \enspace .
\end{equation}
Here, $A_i$ and $B_{i,a}$ are the (possibly infinite) action spaces in state $i$,
$r_i^{a b}$ is a stage payoff and $P_{i j}^{a b}$ gives the transition probability from
state $i$ to state $j$, when action $a$ is chosen by the first player and action $b$
is chosen by the second player.
We have that $P_{i j}^{a b} \geq 0$ and $\sum_{j=1}^n P_{i j}^{a b} = 1$.
This operator governs the evolution of the value function of the game.
We refer the reader to~\cite{FV97,NS03,Sor04} for more background.

Shapley operators are characterized by the following two properties:
they are {\em monotone}, i.e., order-preserving ($\R^n$ being endowed with the standard
partial order), and they are {\em additively homogeneous}, i.e., they commute with
the addition of a constant (a vector with all its entries equal).
The importance of these properties in the theory of dynamic programming was recognized early on,
in particular by Blackwell~\cite{Bla65} and Crandall and Tartar~\cite{CT80}.
For more recent studies, we refer the reader to the work of Rosenberg and Sorin~\cite{RS01a,Sor04}, providing a game-theoretic viewpoint,
and of Mart{\'\i}nez-Legaz, Rubinov and Singer~\cite{MLRS02},
providing an abstract convexity viewpoint. 
Order preserving and additively homogeneous maps are known to be nonexpansive in the sup-norm,
and also in certain {\em weak Minkowski norms}.
Such weak norms have been studied in metric geometry, in particular by Papadopoulos and Troyanov~\cite{PT14}.
They are substitutes of norms that are not necessarily symmetric or coercive,
and arise naturally in the study of Shapley operators, see~\cite{GV12}. 

Conversely, Kolokoltsov showed that every operator from $\R^n$ to $\R^n$
that satisfies the two above properties (monotonicity and additive homogeneity)
can be written as a Shapley operator~\cite{Kol92}.
His proof is based on a minimax representation formula of Evans,
which applies more generally to Lipschitz functions~\cite{Eva84}.
Rubinov and Singer~\cite{RS01b} showed that the transition probabilities in~\eqref{eq:ShapleyOperator}
may even be chosen to be degenerate (deterministic), i.e., such that every stochastic vector
$(P_{i j}^{a b})_{1 \leq j \leq n}$ may be required to have precisely one nonzero entry.  
Gunawardena and Sparrow obtained independently an equivalent result (see~\cite[Prop.~2.3]{Gun03}).
In the context of abstract convexity,
a representation of functions with similar properties over a cone
has been given in~\cite{DMLR04,DMLR08}.

An important subclass of monotone and additively homogeneous operators arises
by requiring every coordinate map $x\mapsto [f(x)]_i$ to be convex.
Operators in this class appear in stochastic control problems.
Then, one can obtain a one-player type representation which has the same form as~\eqref{eq:ShapleyOperator},
but with no infimum, by exploiting the Fenchel-Legendre duality~\cite{AG03}.
In infinite dimension, representation theorems for real functions with the same kind of properties
have been proven in the setting of convex risk measures~\cite{FS02,FRG02}.

Another important subclass is composed of {\em positively homogeneous} operators, i.e.,
those operators which commute with the product by a nonnegative constant.
They appear in Perron-Frobenius theory~\cite{GG04}, in the study of repeated zero-sum games~\cite{RS01a}.
In infinite dimension, convex risk measures satisfying the positive homogeneity property
are referred to as coherent risk measures~\cite{ADEH99,Del02}.

In this paper, we first establish a general minimax representation theorem
which applies to real functions on a topological vector space
that are nonexpansive with respect to a weak Minkowski norm (Theorem~\ref{thm:RepresentationN}). 
We also characterize the nonexpansive maps that are positively homogeneous of degree $1$
(Theorem~\ref{thm:RepresentationHN}).
As a corollary, we arrive at our main application:
a representation theorem for monotone, additively
and positively homogeneous operators from $\R^n$ to $\R^n$ (Corollary~\ref{coro:PaymentFreeOperator})
showing that they correspond precisely to the class of zero-sum games in which the payment occurs only at the last stage (Everett's recursive games). This is
motivated by the ``operator approach'' to zero-sum games~\cite{RS01a}, 
in which properties of such games are derived from axiomatic
properties of their dynamic programming operators: we show
that the dynamic programming operators of the subclass of recursive games
are characterized axiomatically as well.

This representation also leads to an approximation of such operators by polyhedral maps (involving a finite number of $\min$ and $\max$ operations). 
Such approximations can been used in the setting of ``max-plus basis methods'' for the numerical solution of Hamilton-Jacobi type PDE,
see~\cite{bookofmac,AGL08} for background, and the recent work of McEneaney and Pandey~\cite{McEP15} for an application of minimax approximations. 
We also arrive at a representation theorem for nonconvex risk measures.

\section{Preliminary results}
\label{sec:PreliminaryResults}

Throughout the paper, we denote by $V$ a real topological vector space (TVS).
We denote by $V^*$ the dual space of $V$ (that is, the space of linear forms on $V$), by $V'$ its topological dual space (that is, the space of continuous linear forms on $V$), and by $\<\cdot,\cdot>$ the duality product.

We shall specially consider the situation in which $V$ is a vector space with an (Archimedean) order unit.
I.e., we assume that $V$ is a real vector space with an order relation $\leq$ that is compatible with the algebraic structure of $V$, that is, satisfying the following two axioms:
\begin{itemize}
  \item $x \leq y \implies x+z \leq y+z, \enspace \forall z \in V$ ;
  \item $x \leq y \implies \lambda x \leq \lambda y, \enspace \forall \lambda \in \R_+$ ,
\end{itemize}
where $\R_+$ is the set of nonnegative real numbers.
We also assume that $V$ is equipped with a special vector $\unit_V$ (or simply $\unit$
if the context is clear), called an {\em order unit}, and such that,
for every $x \in V$, there exists a scalar $\lambda > 0$ with $x \leq \lambda \unit_V$.
Finally, we assume that this order unit is {\em Archimedean}, i.e., such that,
for any $x \in V$, we have $x \geq 0$ if $\lambda \unit_V + x \geq 0$ for all $\lambda > 0$.
Such a space will be endowed with the topology defined by the following norm:
\begin{equation}
  \label{eq:SeminormUnit}
  \|x\|_{e_V} = \inf \{ \lambda \in \R \mid -\lambda \unit_V \leq x \leq \lambda \unit_V \} \enspace ,
\end{equation}
see~\cite{PT09}.
Note that if $V$ is the Euclidean space $\R^n$ equipped with the standard partial order, then the standard unit vector of $\R^n$ is an order unit, and the corresponding norm~\eqref{eq:SeminormUnit} is the usual sup-norm. Hence, we will abusively refer to~\eqref{eq:SeminormUnit} as the ``sup-norm'' even in general situations.

An important particular case is obtained when $V$ is an AM-space with unit, i.e., a Banach lattice equipped with an order unit, such that the norm satisfies~\eqref{eq:SeminormUnit}, see~\cite{AB06}.
By the Kakutani-Krein theorem, any AM-space with unit is isomorphic (lattice isometric) to a space $\mathcal{C}(K)$ of continuous functions over some compact Hausdorff set $K$, equipped with the sup-norm,
see~\cite[Th.~8.29]{AB06} or \cite[Ch.~II, Th.~7.4]{Sch74}.

Given two vector spaces with an order unit, $(V,\unit_V)$ and $(W,\unit_W)$, we will be interested in maps $f: V \to W$ that satisfy some of the following properties:
\begin{align*}
  && \text{monotonicity:} \quad & x \leq y \implies f(x) \leq f(y) \enspace ; \tag{\textsf{M}} \label{eq:M}\\
  && \text{additive homogeneity:} \quad & f(\lambda \unit_V + x) = \lambda \unit_W + f(x), \enspace \lambda \in \R \enspace ; \tag{\textsf{AH}} \label{eq:AH}\\
  && \text{additive subhomogeneity:} \quad & f(\lambda \unit_V + x) \leq \lambda \unit_W + f(x), \enspace \lambda \in \R_+ \enspace ; \tag{\textsf{ASH}} \label{eq:ASH}\\ 
  && \text{sup-norm nonexpansiveness:} \quad & \|f(x)-f(y)\|_{e_W} \leq \|x-y\|_{e_V} \enspace ; \tag{\textsf{N}} \label{eq:N}\\
  && \text{positive homogeneity:} \quad & f(\lambda x) = \lambda f(x), \enspace \lambda \in \R_+ \enspace . \tag{\textsf{H}} \label{eq:H}
\end{align*}

The importance of the monotonicity and the additive homogeneity properties in optimal control and game theory is well known.
Crandall and Tartar~\cite{CT80} showed that
\[
  \eqref{eq:M} \enspace \text{and} \enspace \eqref{eq:AH} \enspace \iff \enspace \eqref{eq:N} \enspace \text{and} \enspace \eqref{eq:AH} \enspace ,
\]
when $V = W $ is a $L^\infty$ space. 
It is also known that
\[
  \eqref{eq:M} \enspace \text{and} \enspace \eqref{eq:ASH} \enspace \iff \enspace \eqref{eq:M} \enspace \text{and} \enspace \eqref{eq:N} \enspace ,
\]
see, e.g.,~\cite{AG03}.
These relations are readily generalized to any vector space with an order unit.

The monotonicity and additive homogeneity properties turn out to be related with nonexpansiveness in {\em weak Minkowski norms}.  By weak Minkowski norm on $V$, we mean a function $q : V \to \R$ that is convex and positively homogeneous (property~\eqref{eq:H}), but not necessarily symmetric (we do not require that $q(x)=q(-x)$). Our definition is a variant of the one in~\cite{PT14}, where $q$ is also required to be nonnegative and may take infinite values. 
We say that a real function on $V$, $f : V \to \R$, is {\em nonexpansive} with respect to $q$ if $f(x)-f(y) \leq q(x-y)$ for every $x,y \in V$.

When $V=\R^n$, a useful example of weak Minkowski norm, arising in Hilbert geometry~\cite{PT14}, is the ``top'' map $\sft$
defined by
\[
  \sft(z):= \max_{1 \leq i \leq n} z_i
\]
or its variant, 
\[ \sft^+(z) := \max(\sft(z),0)
  \enspace.
\]

When $V=W=\R^n$, we shall consider the following properties:
\begin{align*}
  && \text{$\sft$-nonexpansiveness:} \quad & \sft (f(x)-f(y)) \leq \sft (x-y) \enspace ; \tag{$\mathsf{N}_\sft$} \label{eq:NT}\\
  && \text{$\sft^+$-nonexpansiveness:} \quad & \sft^+ (f(x)-f(y)) \leq \sft^+ (x-y) \enspace . \tag{$\mathsf{N}^+_\sft$} \label{eq:NTP}
\end{align*}
Gunawardena and Keane showed in~\cite{GK95} that
\begin{align*}
  \eqref{eq:M} \enspace \text{and} \enspace \eqref{eq:AH} \enspace & \iff \enspace \eqref{eq:NT} \enspace , \\
  \eqref{eq:M} \enspace \text{and} \enspace \eqref{eq:ASH} \enspace & \iff \enspace \eqref{eq:NTP} \enspace .
\end{align*}
Again, these relations can be readily generalized to any vector space with an order unit, defining, for a vector space $V$ with an order unit $\unit$, the function $\sft$ by
\begin{equation}
  \label{eq:Top}
  \sft(x) = \inf \{ \lambda \in \R \mid x \leq \lambda \unit \} \enspace.
\end{equation}

\section{Representation theorems}

In this section, we give a general minimax characterization of nonexpansive functions with respect to a weak Minkowski norm, which will lead to extension
and refinements of the known minimax representations of Shapley operators.

\subsection{Representation of nonexpansive functions with respect to weak Minkowski norms}

We shall need the following lemma, which is a variation on Legendre-Fenchel duality.
\begin{lemma}
  \label{lem:RepresentationWeakNorm}
  Let $V$ be a real TVS and let $q: V \to \R$ be a weak Minkowski norm, continuous with respect to the topology of $V$.
  Then, for every $x \in V$,
  \begin{equation}
    \label{eq:RepresentationWeakNorm}
    q(x) = \max_{p \in P} \<p,x> \enspace ,
  \end{equation}
  where $P := \{ p \in V^* \mid \forall x \in V, \enspace \<p,x> \leq q(x) \}$ is a nonempty convex set of $V'$, compact for the weak* topology.
\end{lemma}
\begin{proof}
  For $x \in V \setminus \{0\}$ we have, by definition of $P$, $q(x) \geq  \sup_{p \in P} \<p,x>$.
  Let $p$ be the linear form defined on the vector subspace $\R x$ of $V$ generated by $x$, such that $\<p,x>=q(x)$.
  If $\lambda \geq 0$, then $\<p,\lambda x> = \lambda \<p,x> = \lambda q(x) = q(\lambda x)$, since $q$ is positively homogeneous.
  If $\lambda < 0$, then $\<p,\lambda x> = \lambda \<p,x> = - |\lambda| q(x) \leq |\lambda| q(-x) = q(\lambda x)$, the inequality coming from the convexity of the weak norm: $0 = q(0) \leq q(x)+q(-x)$.
  Hence $p$ is dominated by $q$ on the vector subspace $\R x$ and according to the Hahn-Banach extension theorem~\cite[Th.~5.53]{AB06}, there is a linear extension $\hat p$ of $p$ to $V$ that is dominated by $q$ on $V$.
  Therefore, $\hat p \in P$ and $q(x)=\<\hat p,x>$, which shows that $q(x) = \max_{p \in P} \<p,x>$.
  This also proves that $P$ is nonempty.

  Note that, since $q$ is continuous, every linear form $p$ dominated by $q$ is also continuous, hence $P \subset V'$.
  The convexity of $P$ is straightforward and since this set can also be written as $\bigcap_{x \in V} \{ p \in V^* \mid \<p,x> \leq q(x) \}$, we deduce that $P$ is a weak* closed set.
  Furthermore, for every $x \in V$ and every $p \in P$ we have $-q(-x) \leq \<p,x> \leq q(x)$.
  Hence, $P$ is pointwise bounded.
  Applying the Tychonoff Product theorem~\cite[Th.~2.61]{AB06}, we deduce that $P$ is weak* compact.
\end{proof}

\begin{remark}
  In Lemma~\ref{lem:RepresentationWeakNorm}, we did not assume that $V$ is Hausdorff and locally convex.
  When these properties hold, Lemma~\ref{lem:RepresentationWeakNorm} becomes a direct application of~\cite[Th.~7.52]{AB06}, which applies more generally to dual pairs.
  Alternatively, this result can be obtained using the Legendre-Fenchel duality for convex proper lower semicontinuous functions~\cite[Prop.~4.1]{ET99}, still assuming that $V$ is locally convex and Hausdorff.
\end{remark}

The following simple observation shows that the maximum in~\eqref{eq:RepresentationWeakNorm} is attained in the closure of the set of extreme points of $P$.
\begin{lemma}
  \label{lem:RepresentationWeakNorm2}
  Let $q: V \to \R$ be a continuous weak Minkowski norm, and define $P$ as in Lemma~\ref{lem:RepresentationWeakNorm}.
  Denote by $\ext{P}$ the set of extreme points of $P$, and by $\clo{\ext{P}}$ its closure with respect to the weak* topology.
  Then, for every $x \in V$,
  \begin{equation}
    \label{eq:RepresentationWeakNorm2}
    q(x) = \sup_{p \in \ext{P}} \<p,x> = \max_{p \in \clo{\ext{P}}} \<p,x> \enspace .
  \end{equation}
  In particular, if $\ext{P}$ is closed, then the maximum in~\eqref{eq:RepresentationWeakNorm} is attained at an extreme point of $P$.
\end{lemma}

\begin{proof}
  We already showed in Lemma~\ref{lem:RepresentationWeakNorm} that for every $x \in V$, the supremum
  \[ q(x) = \sup_{p\in {P}} \<p,x> \]
  is attained at some point $r \in P$.
  It follows from the Krein-Milman theorem that $P$ is the closed convex hull of the set $\ext{P}$, the closure being understood with respect to the weak* topology.
  Hence, there exists a net $(r_\alpha)$ of elements of the convex hull of $\ext{P}$ that converges to $r$ in this topology.
  In particular, every $r_\alpha$ can be written as a finite sum $r_\alpha = \sum_{1 \leq i \leq m} \lambda_i p_i$ with $\lambda_i \geq 0, \enspace \sum_{1 \leq i \leq m} \lambda_i = 1$, and $p_i \in \ext{P}$, for some $m \geq 1$, and so, $\<r_\alpha,x> = \sum_{1 \leq i \leq m} \lambda_i \<p_i,x> \leq \sup_{p \in \ext{P}} \<p,x>$. 
  We deduce that $q(x) = \lim_\alpha \<r_\alpha,x> \leq \sup_{p \in \ext{P}} \<p,x>$.
  The opposite inequality $q(x) \geq \sup_{p \in \ext{P}} \<p,x>$ follows readily from~\eqref{eq:RepresentationWeakNorm}.
  Using  the weak* continuity of the map $p \mapsto \<p,x>$, we get $\sup_{p \in \ext{P}} \<p,x>= \sup_{p\in\clo{\ext{P}}}\<p,x>$.
  Finally, since $\clo{\ext{P}}$, which is a weak* closed subset of the weak* compact set $P$, is also weak* compact, we deduce that the latter supremum is attained.
\end{proof}

We deduce from the previous lemmas a minimax representation theorem that directly extends the result of Rubinov and Singer~\cite[Th.~5.3]{RS01b}.
\begin{theorem}
  \label{thm:RepresentationN}
  Let $V$ be a real TVS and let $q:V \to \R$ be a weak Minkowski norm, continuous with respect to the topology of $V$.
  Denote by $P := \{ p \in V^* \mid \forall x \in V, \enspace \<p,x> \leq q(x) \}$.
  A function $f:V \to \R$ is nonexpansive with respect to $q$ if, and only if,
  \begin{equation}
    \label{eq:RepresentationN}
    f(x) = \min_{y \in V} \; \max_{p \in P} \; \left\{ \<p,x-y> + f(y) \right\} \enspace , \quad \forall x \in V \enspace .
  \end{equation}
  Moreover, the value of the latter expression does not change if the maximum is restricted to the points $p \in \clo{\ext{P}}$.
\end{theorem}

\begin{proof}
  If $f$ is nonexpansive with respect to $q$, then we have, by definition, $f(x)-f(y) \leq q(x-y)$ for every $x,y \in V$.
  We readily deduce that $f(x) = \min_{y \in V} q(x-y)+f(y)$, the minimum being attained in $x$.
  Replacing $q$ by its expression given in~\eqref{eq:RepresentationWeakNorm}, we get the minimax representation formula~\eqref{eq:RepresentationN}.
  The remaining part of the theorem follows directly from Lemma~\ref{lem:RepresentationWeakNorm2}.
  Conversely, as a consequence of Lemma~\ref{lem:RepresentationWeakNorm}, any real function given by~\eqref{eq:RepresentationN} is easily seen to be nonexpansive with respect to $q$.
\end{proof}

The minimax representation in Theorem~\ref{thm:RepresentationN}
has the following dual maximin version, the minimum and maximum being taken over
the same sets. 
\begin{corollary} \label{thm:RepresentationNdual}
  Let $V$, $q$ and $P$ be as in Theorem~\ref{thm:RepresentationN}.
  A function $f:V \to \R$ is nonexpansive with respect to $q$ if, and only if,
  \begin{equation}
    \label{eq:RepresentationNdual}
    f(x) = \max_{y \in V} \; \min_{p \in P} \; \left\{ \<p,x-y> + f(y) \right\} \enspace , \quad \forall x \in V \enspace .
  \end{equation}
  Moreover, the value of the latter expression does not change if the minimun is restricted to the points $p \in \clo{\ext{P}}$.
\end{corollary}

\begin{proof}
  From the definition, it
  is easy to check that $f$ is nonexpansive with respect to $q$ if and
  only if $-f$ is nonexpansive with respect to $q':x\mapsto q(-x)$ and that 
  $q'$ satisfies Lemma~\ref{lem:RepresentationWeakNorm}
  with the set $P'=-P$ instead of $P$.
  Applying \eqref{eq:RepresentationN} to $-f$, $q'$ and $P'$
  gives~\eqref{eq:RepresentationNdual}.
\end{proof}
Similarly, all the subsequent results admit dual versions.

\begin{remark}\label{rk:convex}
  If $f$ is also assumed to be convex,
  the minimization and maximization in~\eqref{eq:RepresentationN} 
  can be switched to obtain:
  \[     f(x) =  \max_{p \in P} \; \min_{y \in V} \; \left\{ \<p,x-y> + f(y) \right\} \enspace , \quad \forall x \in V \enspace . \] 
  This new representation is equivalent to the property that $f$ is equal
  to $f^{**}$, where $f^*$ denotes the Fenchel-Legendre transform of $f$, 
  together to the one that the subdifferential of $f$ 
  is necessarily included in $P$.
  This leads to a one-player type representation 
  $f(x) =  \max_{p \in P} \; \left\{ \<p,x> - f^*(p) \right\}$, 
  which has the same form  as~\eqref{eq:ShapleyOperator},
  but with no infimum. Representation formul\ae\ of this nature
  have appeared in the theory of convex risk measures~\cite{FS02,Del02} and
  also in the setting of Markov decision processes~\cite{AG03} when $V=\R^n$.
\end{remark}

\begin{remark}\label{rk:moreau}
  As remarked in the proof of Theorem~\ref{thm:RepresentationN},
  $f$ is nonexpansive with respect to $q$ if, 
  and only if, $f(x) = \min_{y \in V} q(x-y)+f(y)$, for all $x\in V$.
  The dual version also applies:
  $f$ is nonexpansive with respect to $q$ if, 
  and only if, $f(x) = \max_{y \in V} -q(y-x)+f(y)$, for all $x\in V$.
  As noted in~\cite[Prop.~6.7]{AGK},
  these conditions can be interpreted in terms of Moreau conjugacies~\cite{moreau70}, meaning that $-f=f^c$ and $f=(-f)^{c'}$, where
  $f^c$ is the Moreau conjugacy of $f$ with respect to the 
  coupling function $c:(x,y)\in V\times V\mapsto -q(x-y)\in \R$,
  and $c'(x,y):=c(y,x)=-q(y-x)$.
  Using these two conditions, one deduces that $f:V\to \R$ is nonexpansive 
  with respect to $q$ if, and only if, $f=f^{cc'}$.
  This leads to a minimax representation which does not have the form of~\eqref{eq:ShapleyOperator}. 
\end{remark}
\begin{corollary}
  \label{coro:RepresentationMAH}
  Let $V$ be a vector space with an order unit $\unit$, and the topology defined by the norm~\eqref{eq:SeminormUnit}.
  Let
  \[
    \Delta := \{ p \in V^* \mid \<p,\unit> = 1, \enspace \<p,x> \geq 0, \enspace \forall x \in V, x \geq 0 \}
    \enspace .
  \]
  Then, a function $f:V \to \R$ is monotone and additively homogeneous if, and only if, 
  \[
    f(x) = \min_{y \in V} \; \max_{p \in \Delta} \; \left\{ \<p,x-y> + f(y) \right\} \enspace .
  \]
  Moreover, the value of the latter expression is not changed if the minimum is restricted to those $y \in V$ such that $f(y)=0$, or if the maximum is restricted to the points $p \in \clo{\ext{\Delta}}$.
\end{corollary}

\begin{proof}
  Let $f$ be a monotone and additively homogeneous real function on $V$.
  As exposed in Section~\ref{sec:PreliminaryResults}, this is equivalent to the function $f$ being nonexpansive with respect to the weak Minkowski norm $\sft$.
  Let $P$ be defined as in Theorem~\ref{thm:RepresentationN} and let us show that $P=\Delta$.
  If $p \in P$, then, for all $x \in V, x \geq 0$, we have $\sft(-x) \leq 0$, so that $\<p,-x> \leq 0$, hence $\<p,x> \geq 0$.
  Moreover, $\sft(\unit) \leq 1$ and $\sft(-\unit) \leq -1$, so that $\<p,\unit> \leq 1$ and $\<p,-\unit> \leq -1$, which shows that $\<p,\unit>=1$.
  This shows that $P\subset \Delta$.
  Conversely, if $p \in \Delta$, then for all $x \in V$ such that $\sft(x) < \lambda$, we have $x \leq \lambda \unit$, so that $\<p,x> = \<p,x-\lambda \unit> + \lambda \<p,\unit> \leq \lambda$.
  This implies that $\<p,x> \leq \sft(x)$ for all $x \in V$, hence $p \in P$.
  Applying  Theorem~\ref{thm:RepresentationN}, we get the first equality in the corollary.
  Now, remark  that $\<p,x-y> + f(y) = \<p,x-z>$ with $z = y-f(y)\unit$, since $\<p,\unit>=1$ and $f(z)=f(y)-f(y)=0$.
  Then, a change of variable leads to
  \[
    f(x)= \min_{\substack{y \in V \\ f(y)=0}} \; \max_{p \in \Delta} \; \<p,x-y> \enspace .
  \]
  The remaining part of the corollary follows from Theorem~\ref{thm:RepresentationN} and the converse is straightforward.
\end{proof}

\begin{example}
  If $V = \R^n$ and $f$ is monotone and additively homogeneous, then, 
  as recalled in Section~\ref{sec:PreliminaryResults}, it is nonexpansive in the weak Minkowski norm $\sft$.
  Then, the representation result of Rubinov and Singer~\cite[Th.~5.3]{RS01b},
  which shows that
  \[
    [f(x)]_i = \min_{y\in \R^n} \max_{1\leq j\leq n} \{x_j -y_j + [f(y)]_i\} \enspace ,
  \]
  is a special case of Corollary~\ref{coro:RepresentationMAH}.
\end{example}

\subsection{Representation of positively homogeneous nonexpansive functions}

We now consider nonexpansive functions that are positively homogeneous.
The following theorem characterizes these functions as minimax of nonexpansive linear forms. 

\begin{theorem}
  \label{thm:RepresentationHN}
  Let $V$ be a real TVS and let $q:V \to \R$ be a weak Minkowski norm, continuous with respect to the topology of $V$.
  Denote by $P := \{ p \in V^* \mid \forall x \in V, \enspace \<p,x> \leq q(x) \}$.
  A function $f:V \to \R$ is positively homogeneous and nonexpansive with respect to $q$ if, and only if,
  \begin{align}
    \label{eq:RepresentationHN}
    f(x) = \min_{y \in V} \; \max_{p \in P_y} \; \<p,x> \enspace ,
  \end{align}
  where $P_y := \{ p \in P \mid \<p,y> \leq f(y) \}$.
  Moreover, the value of the expression~\eqref{eq:RepresentationHN} does not change if the maximum is restricted to the points $p \in \clo{\ext{P_y}}$.
\end{theorem}

\begin{proof}
  The sufficiency of the condition is straightforward, so we only prove that it is necessary.
  Let $f$ be a real function positively homogeneous and nonexpansive with respect to $q$.
  As a direct consequence of nonexpansiveness, we get that $f(x) = \min_{y \in V} f(y)+q(x-y)$.
  By a change of variable, we have $f(x) = \min_{y \in V} \inf_{\lambda > 0} \lambda f(y)+q(x-\lambda y)$, since $f$ is also positively homogeneous.
  There, the minimum in $y$ is attained at all $\mu x$ with $\mu > 0$.
  Given $y \in V$, let $g_y$ be the function defined on $V$ by $g_y:x \mapsto \inf_{\lambda > 0} \lambda f(y)+q(x-\lambda y)$, so that $f(x)= \min_{y \in V} g_y(x)$.
  We next show that $g_y$ is a continuous weak Minkowski norm.

  Firstly, $g_y$ is bounded above by $q$ (take $\lambda \to 0$) and below by $f$ (since $f = \min_{y} g_y$).
  In particular, $g_y$ is bounded above on a neighborhood of each point of $V$.
  Secondly, using the positive homogeneity of $f$ and $q$, we check that $g_y$ also shares this property.
  Thirdly, $g_y$ is convex.
  Indeed, the function $x \mapsto f(y)+q(x-y)$ is convex because so is $q$.
  Its perspective function, defined on $\R \times V$ by $(\lambda,x) \mapsto \lambda f(y)+q(x-\lambda y)$ if $\lambda>0$ and $+\infty$ otherwise, is also convex~\cite[Prop.~8.23]{BC11}.
  The conclusion follows from the fact that $g_y$ is the marginal function with respect to the first variable of this former function.

  We have shown that $g_y$ is a positively homogeneous convex function, finite everywhere and bounded above on a neighborhood of each point of $V$.
  Hence, it is also continuous on $V$~\cite[Th.~5.43]{AB06} and we deduce from Lemma~\ref{lem:RepresentationWeakNorm} that it is the support function of the weak* compact convex set $Q_y := \{ p \in V^* \mid \forall x \in V, \enspace \<p,x> \leq g_y(x) \}$.
  To conclude, it remains to show that this set is $P_y$.

  Let $p \in Q_y$.
  Then, for every $x \in V$ and every $\lambda > 0$, we have $\<p,x> \leq \lambda f(y) + q(x-\lambda y)$.
  Taking $\lambda \to 0$ we deduce that $p \in P$, and taking $x=y$ and $\lambda = 1$ we deduce that $\<p,y> \leq f(y)$.
  Hence, $p \in P_y$ which shows that $Q_y \subset P_y$.

  Consider now $p \in P_y$.
  For every $x \in V$ and every $\lambda > 0$ we have
  \begin{equation*}
    \<p,x> = \<p,x> - \<p,\lambda y> + \<p,\lambda y> = \<p,x-\lambda y> + \lambda \<p,y> \leq q(x-\lambda y) + \lambda f(y) \enspace .
  \end{equation*}
  Taking the infimum for all $\lambda > 0$ in the right-hand side of the last inequality, we get that $\<p,x> \leq g_y(x)$.
  Hence $p \in Q_y$ which shows that $Q_y \subset P_y$ and consequently that $Q_y = P_y$.
\end{proof}

The following is an immediate corollary. 
\begin{corollary}
  \label{coro:RepresentationMAH2}
  Let $V$ be a vector space with an order unit $\unit$, and the topology defined by the norm~\eqref{eq:SeminormUnit}. 
  Then, a function $f:V \to \R$ is monotone, additively homogeneous, and positively homogeneous if, and only if,
  \begin{align}
    f(x) = \min_{y \in V} \max_{p \in \Delta_y} \; \<p,x> \enspace ,
    \label{eq:RepresentationHN2}
  \end{align}
  where 
  \[
    \Delta_y := \{p \in V^* \mid \<p,y> \leq f(y), \enspace \<p,\unit>=1, \enspace \<p,x> \geq 0, \enspace \forall x \in V, x \geq 0 \} \enspace .
  \]
  Moreover, the value of the expression in~\eqref{eq:RepresentationHN2} does not change if the minimum is restricted to those $y \in V$ such that $f(y)=0$, or if the maximum is restricted to the points $p \in \clo{\ext{\Delta_y}}$.
\end{corollary}

\section{Applications}

We now point out some applications of the present representation theorem
to nonconvex risk measure and to the representation of recursive games. 

\subsection{Representation of nonconvex risk measures}

Let $(\Omega, \mathcal{F}, \bbP)$ be a probability space.
Denote by $L^\infty(\bbP)$ the space of equivalence classes of a.s.\ bounded real-valued random variables, equipped with the usual $L^\infty$ norm,
and the usual partial order, i.e., $X \leq Y$ if any representative of $Y-X$ is a random
variable almost surely nonnegative.
This is a special case of AM-space with unit, the unit $\unit$ being the random variable a.s.\ equal to $1$.
Its topological dual space is the space of finitely additive measures of bounded variation which are absolutely continuous with respect to $\bbP$, denoted by $\ba(\bbP)$~\cite[Th.~IV.8.16]{DS88}.
Following a standard notation, we will write $\E_p[X]$ instead of $\<p,X>$ for $X \in L^\infty(\bbP)$ and $p \in \ba(\bbP)$.
We denote by $\ba^+(\bbP)$ the set of positive bounded finitely additive measures and by $\Delta(\bbP) := \{ p \in \ba^+(\bbP) \mid \E_p[\unit]=1 \}$ the set of finitely additive probability measures.

A {\em risk measure} is a function $\mu: L^\infty(\bbP) \to \R$ that satisfies the following conditions, for every $X,Y \in L^\infty(\bbP)$:
\begin{itemize}
  \item $X \leq Y \implies \mu(X) \geq \mu(Y)$ ;
  \item $\mu(X+\lambda \unit) = \mu(X)-\lambda, \enspace \forall \lambda \in \R$ ;
\end{itemize}
see~\cite{FS11}.
This is equivalent to the function $\rho:X \mapsto \mu(-X)$ being monotone and additively homogeneous.

We say that a risk measure $\mu$ is {\em coherent} if it is also convex and positively homogeneous, see~\cite{ADEH99,Del02}.
It is known that a coherent risk measure can be represented in the form
\[ \mu(X) = \sup_{p \in \mathcal Q_\ba} \E_p[-X] \]
where $\mathcal Q_\ba$ is a convex subset of the space of finitely additive probability measures on $\Omega$, closed with respect to the $\sigma(\ba(\bbP),L^\infty(\bbP))$-topology.
As a direct application of Corollary~\ref{coro:RepresentationMAH}, we obtain a similar representation for general nonconvex risk measures.
\begin{corollary}
  \label{coro:RepresentationRiskMeasure}
  Let $\mu$ be a risk measure on $L^\infty(\bbP)$.
  Then,
  \[
    \mu(X) = \min_{ \substack{Y \in L^\infty(\bbP) \\ \mu(Y)=0} } \; \max_{p \in \Delta(\bbP)} \; \E_p[Y-X] \enspace .
  \]
\end{corollary}

The following corollary characterizes the nonconvex risk measures that are positively homogeneous.
It is a direct application of Corollary~\ref{coro:RepresentationMAH2}.
\begin{corollary}
  Let $\mu$ be a positively homogeneous risk measure.
  Then,
  \[
    \mu(X) = \min_{ \substack{Y \in L^\infty(\bbP) \\ \mu(Y)=0} } \; \max_{ \substack{p \in \Delta(\bbP) \\ \E_p[Y] \geq 0} } \; \E_p[-X] \enspace.
  \]
\end{corollary}

\subsection{Representation of payment-free Shapley operators}

Here, we consider the vector space $V=\R^n$, $\<\cdot,\cdot>$ denoting its usual scalar product.
We call {\em payment-free} Shapley operator, an operator over $\R^n$ that is monotone, additively and positively homogeneous. This terminology is justified by the following corollary, which shows that all operators of this kind precisely arise from zero-sum games
in which the function $(a,b)\mapsto r_i^{ab}$ representing the payment made
at every stage is identically zero. Such games, in which the payment only occurs
the ``last day'', have been studied after Everett under the name of {\em recursive games}~\cite{Eve57}. The operators of these games also arise as recession functions of general Shapley operators. They allow one to study mean payoff
and ergodicity problems, see~\cite{RS01a} and~\cite{AGH15}. 
\begin{corollary}
  \label{coro:PaymentFreeOperator}
  Let $F:\R^n \to \R^n$ be a payment-free Shapley operator, i.e.\ a monotone, additively and positively homogeneous operator $\R^n \to \R^n$.
  Then, $F$ can be represented as
  \begin{equation}
    \label{eq:PaymentFreeOperator}
    F_i(x) = \min_{a \in A_i} \; \max_{b \in B_{i,a}} \; \sum_{j=1}^n P_{i j}^{a b} x_j, \quad \forall x \in \R^n, \quad \forall i \in \{1,\dots,n\}
  \end{equation}
  where: $A_i = \{a \in \R^n \mid F_i(a)=0 \}$; for every $i \in \{1,\dots,n\}$ and every $a \in A_i$, $B_{i,a}$ is a finite subset of $\R^n$; for every $a \in A_i$ and every $b \in B_{i,a}$, $P_i^{a b} = (P_{i j}^{a b})_{1 \leq j \leq n}$ is a stochastic vector with at most two positive coordinates.
\end{corollary}

\begin{proof}
  Any operator of the form~\eqref{eq:PaymentFreeOperator} is a payment-free Shapley operator.
  Conversely, let $F$ be a payment-free Shapley operator.
  Each coordinate function is monotone, additively and positively homogeneous.
  Then, it follows from Corollary~\ref{coro:RepresentationMAH2} that, for every $i \in \{1,\dots,n\}$ and every $x \in \R^n$,
  \begin{equation}
    \label{eq:PaymentFreeOperatorproof}
    F_i(x) = \min_{ \substack{y \in \R^n \\ F_i(y)=0} } \; \max_{ \substack{p \in \Delta_n \\ \<p,y> \leq 0} } \; \<p,x> \enspace .
  \end{equation}
  where $\Delta_n$ is the standard simplex of $\R^n$.

  Let $A_i = \{ a \in \R^n \mid F_i(a)=0 \}$; for $a \in A_i$, let $B_{i,a}$ be the set of extreme points of the polytope $\{p \in \Delta_n \mid \<p,a> \leq 0\}$ (where the maximum in~\eqref{eq:PaymentFreeOperatorproof} is attained); and for $a \in A_i$ and $b \in B_{i,a}$, let $P_i^{a b} = b$.
  Rewriting equation~\eqref{eq:PaymentFreeOperatorproof} with those notations we get exactly~\eqref{eq:PaymentFreeOperator}.

  Finally, a standard result of convex geometry shows that every extreme point of the intersection
  of a polytope with a half-space is either an extreme point of the polytope, or the
  convex combination of two extreme points of this polytope
  (see for instance Lemma~3 of~\cite{FP96}).
  It follows that every element in $B_{i,a}$ is either an extreme point of $\Delta_n$, either a convex combination of two extreme points of $\Delta_n$.
\end{proof}

\subsection{Approximation of Shapley operators}

We use the latter result to approximate payment-free Shapley operators over $\R^n$ by minimax maps where the $\min$ and $\max$ operators are taken over finite sets.
Such maps play an important role algorithmically, in max-plus finite element method~\cite{AGL08},
and more generally in idempotent methods~\cite{McE11,McEP15}. 

Let $q: \R^n \to \R$ be a weak Minkowski norm.
We say that a subset $A \subset \R^n$ is an $\varepsilon$-net of the set $K \subset \R^n$ with respect to (the symmetrization of) $q$ if
\[
  \inf_{a \in A} \max \{ q(x-a),q(a-x) \} < \varepsilon \enspace , \quad \forall x \in K \enspace .
\]
Note that here, since the dimension is finite, $q$ is continuous, and then it is always possible to find a finite $\varepsilon$-net of a compact set with respect to $q$.
\begin{proposition}
  Let $q: \R^n \to \R$ be a weak Minkowski norm and let $f: \R^n \to \R$ be a map nonexpansive with respect to $q$.
  Then, for every compact set $K \subset \R^n$, and for every finite $\varepsilon$-net $(y_\ell)_{1 \leq \ell \leq m}$ of $K$ with respect to $q$, the function
  \[
    g(x) = \min_{1 \leq \ell \leq m} \{ f(y_\ell) + q(x-y_\ell) \}
  \]
  is such that
  \[
    f(x) \leq g(x) \leq f(x) + 2 \varepsilon \enspace, \quad \forall x \in K \enspace .
  \]
\end{proposition}

\begin{proof}
  Let $x \in K$.
  Since $f$ is nonexpansive with respect to $q$, we have $f(x) \leq f(y_\ell) + q(x-y_\ell)$ for every $\ell \in \{1,\dots,m\}$.
  Thus, $f(x) \leq g(x)$.
  We also know that there is an index $\ell_0$ such that $\max \{ q(x-y_{\ell_0}),q(y_{\ell_0}-x) \} < \varepsilon$.
  Hence,
  \[
    f(y_{\ell_0}) + q(x-y_{\ell_0}) \leq f(x) + q(y_{\ell_0}-x) + q(x-y_{\ell_0}) \leq f(x) + 2 \varepsilon \enspace ,
  \]
  from which the second inequality follows.
\end{proof}

\begin{corollary}
  Let $F: \R^n \to \R^n$ be a payment-free Shapley operator.
  Then, for all $\varepsilon > 0$, there exists a payment-free Shapley operator $G: \R^n \to \R^n$ such that
  \[
    F_i(x) \leq G_i(x) \leq F_i(x) +\varepsilon \|x\|_\infty
  \]
  for all $x \in \R^n$ and $1 \leq i \leq n$, which can be represented in the form
  \begin{align*}
    \label{eq:PaymentFreeOperatorApprox}
    G_i(x) = \min_{a \in A_i} \; \max_{b \in B_{i,a}} \; \sum_{j=1}^n P_{i j}^{a b} x_j, \quad \forall x \in \R^n, \quad \forall i \in \{1,\dots,n\}\enspace ,
  \end{align*}
  where $A_i$ and $B_{i,a}$ are {\em finite} sets, and every $P_i^{a b} = (P_{i j}^{a b})_{1 \leq j \leq n}$ is a stochastic vector with at most two positive coordinates.
\end{corollary}

\begin{proof}
  Let $(z_\ell)_{1 \leq \ell \leq m}$ be an $\varepsilon / 2$-net of the unit sphere of $\R^n$
  with respect to the sup-norm.
  Using the minimax representation of $F$, steming from Corollary~\ref{coro:PaymentFreeOperator},
  we know that for all $i \in \{1,\dots,n\}$ and $\ell \in \{1,\dots,m\}$, there exists an action
  $a_{i \ell} \in A_i$ for which the minimum is attained in formula~\eqref{eq:PaymentFreeOperator}
  applied to $F_i(z_\ell)$.
  Let $A^*_i := \{a_{i \ell} \mid 1 \leq \ell \leq m \}$ be the finite subset of $A_i$
  containing all the latter optimal actions in state $i \in \{1,\dots,n\}$.
  Then, let $G: \R^n \to \R^n$ be the payment-free Shapley operator the $i$th coordinate map
  of which is given by
  \begin{equation*}
    G_i(x) := \min_{a \in A^*_i} \; \max_{b \in B_{i,a}} \; \sum_{j=1}^{n} P_{i j}^{a b} x_j
    \enspace , \quad \forall x \in \R^n \enspace ,
  \end{equation*}
  where the action spaces $B_{i,a}$ are the same as in the minimax
  representation~\eqref{eq:PaymentFreeOperator} of $F_i$.
  In particular, we know that they are finite and that all the vectors
  $P_i^{a b} = (P_{i j}^{a b})_{1 \leq \ell \leq m}$ are stochastic, with at most
  two positive coordinates.

  By construction, we have $F(x) \leq G(x)$ for all vectors $x \in \R^n$,
  with equality for every $z_\ell$.
  Now, given a vector $x$ in the unit sphere of $\R^n$, we can choose a vector $z_\ell$
  such that $\|x-z_\ell\|_\infty \leq \varepsilon / 2$. 
  Since both $F$ and $G$ are nonexpansive with respect to the sup-norm, we deduce that
  for all $i \in \{1,\dots,n\}$,
  \[
    G_i(x) - F_i(x) \leq G_i(z_\ell) + \varepsilon / 2 - F_i(x) =
    F_i(z_\ell) - F_i(x) + \varepsilon / 2 \leq \varepsilon \enspace .
  \]
  The conclusion follows from the positive homogeneity of $F$ and $G$.
\end{proof}

\bibliographystyle{amsalpha}
\bibliography{references}

\end{document}